\documentclass[a4paper,reqno,11pt]{amsart}
\usepackage{amsmath}
\usepackage{amssymb}
\usepackage{cases}
\allowdisplaybreaks[4]

\newtheorem{thm}{Theorem}[section]
\newtheorem{prop}[thm]{Proposition}

\newtheorem{lem}[thm]{Lemma}

\newtheorem{cor}[thm]{Corollary}

\newtheorem{conj}[thm]{Conjecture}

\theoremstyle{definition}
\newtheorem{definition}[thm]{Definition}
\newtheorem{definition-proposition}[thm]{Definition-Proposition}
\newtheorem{example}[thm]{Example}

\theoremstyle{remark}
\newtheorem{remark}[thm]{Remark}

\numberwithin{equation}{section}

\newcommand{\bP}{\mathbb{P}}

\newcommand\OO{{\mathcal{O}}}

\newcommand{\bC}{{\mathbb C}}

\newcommand\FF{{\mathcal{F}}}
\newcommand\HH{{\mathcal{H}}}
\newcommand\GG{{\mathcal{G}}}

\newcommand\LL{\mathcal{L}}

\newcommand\CC{{\mathbb{C}}}

  \newcommand\EE{{\mathcal{E}}} 
  \newcommand\RR{{\mathcal{R}}} 
  \newcommand\SSS{{\mathcal{S}}}
 \newcommand{\bF}{{\mathbb F}}

  \newcommand\supp{{\rm{supp}}}

\newcommand\Ext{{\rm{Ext}}}
\newcommand\Hom{{\rm{Hom}}}
 \newcommand\GL{{\rm{GL}}}
 \newcommand\Aut{{\rm{Aut}}}
 \newcommand\RHom{\mathbb{R}{\rm{Hom}}}
\newcommand\Auteq{{\rm{Auteq}}}
 \newcommand\id{{\rm{id}}}
\newcommand\Coh{{\rm{Coh}}}

\begin{document}

\title{Torsion exceptional sheaves on weak del Pezzo surfaces of Type A}
\date{\today}
\author{Pu Cao}
\address{%
Graduate School of Mathematical Sciences, The University of Tokyo
3-8-1 Komaba, Meguro, Tokyo, 153-8914 Japan}
\email{pucao@ms.u-tokyo.ac.jp}
\author{Chen Jiang}
\address{%
Kavli IPMU (WPI), UTIAS, The University of Tokyo, Kashiwa, Chiba 277-8583, Japan.}
\email{chen.jiang@ipmu.jp}
 \thanks{The first author
 was supported by Government Scholarship and the Program for Leading Graduate  Schools, MEXT, Japan. The second author was supported by JSPS KAKENHI Grant Number JP16K17558, the Program for Leading Graduate Schools, and World Premier International Research Center Initiative (WPI), MEXT, Japan}

\begin{abstract}
We investigate torsion exceptional sheaves on a weak del Pezzo surface of degree greater than two whose anticanonical model has at most $A_n$-singularities. We show that every torsion exceptional sheaf can be obtained from a line bundle on a $(-1)$-curve by spherical twists.
\end{abstract} 
\subjclass[2010]{13D09, 14J26}
\keywords{derived categories, exceptional objects, weak del Pezzo surfaces}
\maketitle
\pagestyle{myheadings} \markboth{\hfill  P. Cao \& C. Jiang \hfill}{\hfill Torsion exceptional sheaves on weak del Pezzo surfaces of Type A\hfill}

\section{Introduction}
\noindent We work over the complex number field $\bC$. Let $X$ be a smooth projective variety and $D(X):= D^b(\Coh X)$ the bounded derived category of coherent sheaves on $X$. The category $D(X)$ carries a lot of geometric information on $X$ and has drawn a lot of interest in the study of algebraic varieties. An object $\alpha \in D(X)$ is called \emph{exceptional} if 
$$
\Hom (\alpha,\alpha[i])
\cong
\begin{cases} \bC & i= 0;\\
                                       0         & i \ne 0. 
\end{cases}
$$
Exceptional objects are related to semi-orthogonal decompositions of derived categories and appear in many contexts (see, for example, \cite{Huybrechts}). Hence it is natural to consider the
classification of exceptional objects.

Exceptional objects on {\it del Pezzo surfaces} (i.e., smooth projective surfaces with ample anticanonical bundles) were investigated by Kuleshov and Orlov in \cite{KO} where it is proved that any exceptional object on a del Pezzo surface is isomorphic to a shift of an exceptional vector bundle or a line bundle on a $(-1)$-curve. 

As exceptional objects on del Pezzo surfaces are well-understood, it is natural to consider {\it weak del Pezzo surfaces} (i.e., smooth projective surfaces with nef and big anticanonical bundles). In this case something interesting happens since spherical twist functors (see Definition \ref{def spherical}) are involved due to the existence of $(-2)$-curves on weak del Pezzo surfaces. We could not expect that exceptional objects are as simple as those on del Pezzo surfaces (see Section \ref{section example}), but still we expect that they are so after acting by autoequivalences of the derived category. 
\begin{conj}[{cf. \cite[Conjecture 1.3]{OU}}]\label{conj}Let $X$ be a weak del Pezzo surface.
 For any exceptional object $\EE\in D(X)$, there exists an autoequivalence
$\Phi \in \Auteq(D(X))$ such that $\Phi(\EE)$ is an exceptional vector bundle, or a line bundle on a $(-1)$-curve on $X$.
\end{conj}
Recently, Okawa and Uehara \cite{OU} considered the Hirzebruch surface $\bF_2$, the simplest weak del Pezzo surface. They classified exceptional sheaves on $\bF_2$ and confirmed Conjecture \ref{conj} for those sheaves.
Note that on $\bF_2$, there is no torsion exceptional sheaf due to the absence of $(-1)$-curves. Motivated by Okawa--Uehara's work and this observation, we are interested in torsion exceptional sheaves (and objects) on weak del Pezzo surfaces.
In \cite{C}, the first author treated torsion exceptional sheaves and objects on a weak del Pezzo surface of degree greater than one whose anticanonical model has at most $A_1$-singularities.

On the other hand, one may compare Conjecture \ref{conj} to \cite[Proposition 1.6]{Ishii-Uehara}, where Ishii and Uehara showed that a spherical object on the minimal resolution of an $A_n$-singularity on a surface can be obtained from a line bundle on a $(-2)$-curve by autoequivalence. But the situation for torsion exceptional objects seems to be more complicated since its scheme theoretic support might be non-reduced (see Example \ref{example 2}) while the support of such a spherical object is always reduced (see \cite[Corollary 4.10]{Ishii-Uehara}).

In this article, we give an affirmative answer to Conjecture \ref{conj} for torsion exceptional sheaves on weak del Pezzo surfaces of degree greater than two of {\it Type A} (i.e., those whose anticanonical model has at most $A_n$-singularities). Namely, we prove the following theorem.
\begin{thm}\label{main1}
Let $X$ be a weak del Pezzo surface of degree $d>2$ of Type A, and $\EE$ a torsion exceptional sheaf on $X$. Then there exist a $(-1)$-curve $D$, an integer $d$, and a sequence of spherical twist functors $\Phi_1,\ldots, \Phi_n$ associated to line bundles on chains of $(-2)$-curves such that 
$$
\EE\cong \Phi_1 \circ \cdots \circ \Phi_n(\OO_D(d)).
$$
\end{thm}

\begin{remark}
To be more concisely, Theorem \ref{main1} implies that every torsion exceptional sheaf on such $X$ is a line bundle on a $(-1)$-curve up to the action by $\text{Br}(X)$, where $\text{Br}(X)\subset\Auteq(D(X))$ is the group generated by all spherical twists functors. It is worth to mention that every spherical twist in $\text{Br}(X)$ can be decomposed into twists or inverse twists  associated to line bundles on $(-2)$-curves by Ishii and Uehara \cite{Ishii-Uehara}. The point of Theorem \ref{main1} is that we show that the torsion exceptional sheaf can be obtained from a line bundle on a $(-1)$-curve by acting spherical twist functors associated to line bundles on chains of $(-2)$-curves successively which can be constructed explicitly according to the proof, and no inverse spherical twist functors are involved here.
\end{remark}

In fact, we can prove the following slightly general theorem on torsion exceptional sheaves.
\begin{thm}\label{main2}
Let $X$ be a smooth projective surface and $\EE$ a torsion exceptional sheaf on $X$. 
Assume the following conditions hold:
\begin{enumerate}
\item $\supp(\EE)$ only contains  one $(-1)$-curve $D$ and $(-2)$-curves;
\item the restriction of $\EE$ to $D$ is a line bundle (see Definition \ref{def restriction});
\item $(-2)$-curves in $\supp(\EE)$ forms a disjoint union of $A_n$-configurations with $n\leq 6$;
\item the intersection number of $D$ with any chain of $(-2)$-curves in $\supp(\EE)$ is at most one.
\end{enumerate}
Then there exist an integer $d$, and a sequence of spherical twist functors $\Phi_1,\ldots, \Phi_n$ associated to line bundles on chains of $(-2)$-curves such that 
$$
\EE\cong \Phi_1 \circ \cdots \circ \Phi_n(\OO_D(d)).
$$
\end{thm}

The idea of the proof is based on the observation that, under some good conditions, we can ``factor" a spherical sheaf out of a torsion exceptional sheaf to get another one (see Lemmas \ref{lem EES} and \ref{lemma main2}), and this step actually corresponds to a spherical twist functor. After this factorization, we get an exceptional sheaf with smaller support. Iterating this process,  eventually we get an exceptional sheaf supported on a $(-1)$-curve. To check the conditions that allow us to factor out spherical sheaves, we need a detailed classification of certain torsion rigid sheaves supported on $(-2)$-curves (see Section \ref{section rigid -2}). 

We expect that this idea also works for torsion exceptional sheaves on arbitrary weak del Pezzo surfaces. At least, Lemma \ref{lem EES} is very useful for cutting down the support of the torsion exceptional sheaf. However, to find an appropriate way to apply Lemma \ref{lem EES}, our proof is based on the classification of torsion rigid sheaves supported on $(-2)$-curves (see Section \ref{section rigid -2}). If the configuration of the support gets complicated, then the classification becomes tedious. 
So we hope that one could replace the latter part by more systematical method which can avoid tedious classifications and works for arbitrary weak del Pezzo surfaces.

\medskip
\noindent{\bf Notation and Conventions.} 
Let $X$ be a smooth projective surface. For a coherent sheaf $\EE$ on $X$, we denote by $\supp(\EE)$ 
the support of $\EE$ with reduced induced scheme structure. 
For $\EE, \FF\in D(X)$, we denote $h^i(\EE, \FF):=\dim \Ext^i(\EE, \FF)=\dim \Hom(\EE, \FF[i]),$
and the Euler characteristic $\chi(\EE, \FF):=\sum_{i}(-1)^ih^i(\EE, \FF).$
A $(-1)$-curve (resp. $(-2)$-curve) is a smooth rational curve on $X$ with self-intersection number $-1$ (resp. $-2$).
We say $Z=C_1\cup\cdots\cup C_n$ is a chain of $(-2)$-curves on $X$ if $C_i$ is a $(-2)$-curve and 
$$
C_i\cdot C_j=\begin{cases}
    1 &|i-j|=1;\\
    0 &|i-j|>1.
   \end{cases}
$$
We regard $Z$ as a closed subscheme of $X$ with respect to the reduced induced structure. Sometimes we also regard $Z$ as its fundamental cycle $\sum_iC_i$. For a coherent sheaf $\RR$ on $Z$, we denote by $\deg_{C_l} \RR$ the degree of the restriction $\RR|_{C_l}$ on $C_l\cong \bP^1$. We denote by $\RR_0 =\OO_{C_1\cup \cdots \cup C_n}(a_1,\ldots,a_n)$
the line bundle on $Z$ such that $\deg_{C_l} \RR_0 = a_l$ for all $l$. Sometimes we also consider 
$
\RR_1 =\OO_{r_1C_1\cup \cdots \cup r_nC_n}(a_1,\ldots,a_n)
$
for $r_l\in \{1,2\}$ for all $l$. Here $\RR_1$ is the line bundle on $r_1C_1\cup \cdots \cup r_nC_n$ such that $\deg_{C_l} \RR_1 = a_l$ for all $l$. In other words, $\RR_1|_{r_lC_l}\cong \OO_{r_lC_l}(a_l)$, where $\OO_{2C}(a)$ is the unique non-trivial extension of $\OO_{C}(a)$ by $\OO_{C}(a+2)$ for a $(-2)$-curve $C$ on $X$.

\section{Preliminaries}

\subsection{Exceptional and spherical objects}
We recall the definition of exceptional and spherical objects.
\begin{definition}
Let
$X$ be a smooth projective variety. 
An object $\alpha \in D(X)$ is \emph{exceptional} if 
$$
\Hom (\alpha,\alpha[i])
\cong
\begin{cases} \bC & i= 0;\\
                                       0         & i \ne 0. 
\end{cases}
$$
\end{definition}

\begin{example}
\begin{enumerate}
\item Let $X$ be a smooth projective variety with $H^i(X, \OO_X)=0$ for $i>0$ (e.g. Fano manifolds). Then every line bundle on $X$ is an exceptional object.
\item Let $X$ be a smooth projective surface and $C$ a $(-1)$-curve on $X$.  Then any line bundle on $C$ is an exceptional object.
\end{enumerate}
\end{example}
\begin{definition}[\cite{ST}]\label{def spherical}
Let
$X$ be a smooth projective variety. 

\begin{enumerate} 
\item An object $\alpha \in D(X)$ is \emph{spherical} if 
 $\alpha \otimes \omega_{X} \cong \alpha$ and
$$
\Hom (\alpha,\alpha[i])
\cong
\begin{cases}  
 \bC & i=0,\dim X;\\
0 & i\ne 0,\dim X.                     
\end{cases}
$$

\item Let $\alpha\in D(X)$ be a spherical object. 
Consider the mapping cone 
$$
\mathcal{C}={\rm Cone}(\pi_1^*\alpha^\vee\otimes \pi_2^*\alpha \to \mathcal{O}_{\Delta})
$$
of the natural evaluation
$
\pi_1^*\alpha^\vee\otimes \pi_2^*\alpha \to \mathcal{O}_{\Delta}
$,
where
$
\Delta \subset X \times X
$
is the diagonal and $\pi_i$ is the projection from
$X\times X$ to the $i$-th factor. Then the integral functor 
$T_{\alpha}:=\Phi^{\mathcal{C}}_{X\to X}$ defines an autoequivalence of $D(X)$, 
called the \emph{twist functor} associated to the spherical object $\alpha$. 
By definition, for $\beta\in D(X)$, there is an exact triangle
$$
\RHom(\alpha,\beta)\otimes
\alpha\xrightarrow[]{\textrm{evaluation}}
\beta \to
T_{\alpha}\beta.
$$
\end{enumerate}
%
%
%
\end{definition}
\begin{example}[cf. {\cite[Example 4.7]{Ishii-Uehara}}]
Let $X$ be a smooth projective surface and $Z$ a chain of $(-2)$-curves. Then any line bundle on $Z$ is a spherical object in $D(X)$.
\end{example}

\subsection{Rigid sheaves}In this subsection, we assume that $X$ is a  smooth projective surface. All sheaves are considered to be coherent on $X.$

A coherent sheaf $\RR$ is said to be {\it rigid} if $h^1(\RR, \RR) = 0$.

Kuleshov \cite{Kuleshov} systematically investigated rigid sheaves on surfaces with
  anticanonical class without base components. We collect some interesting properties for rigid sheaves in this subsection for applications. We will use the following easy lemma without mention.
\begin{lem}\label{lem ext non-trivial} Consider an extension of coherent sheaves
$$0 \to \GG_2  \to \RR \to \GG_1 \to 0$$
such that $\RR$ is rigid. Then $h^1(\GG_1, \GG_2)>0$ if and only if this extension is non-trivial.
\end{lem}

\begin{proof}
The `if' part is trivial. For the `only if' part, assume that $h^1(\GG_1, \GG_2)>0$ and this extension is trivial, then 
$\RR\cong \GG_1\oplus \GG_2,$
which implies that 
$
h^1(\RR, \RR)\geq h^1(\GG_1, \GG_2)>0,
$
a contradiction.
\end{proof}
We have the following Mukai's lemma for rigid sheaves.
\begin{lem}[Mukai's lemma, {\cite[2.2 Lemma]{KO}}]\label{lem GEG}For each exact sequence
$$0 \to \GG_2  \to \RR \to \GG_1 \to 0$$
of coherent sheaves such that $$h^1(\RR,\RR)=h^0(\GG_2, \GG_1) = h^2 (\GG_1, \GG_2) = 0,$$
the following hold:
\begin{enumerate}
\item $h^1(\GG_1, \GG_1)=h^1(\GG_2, \GG_2)=0$;
\item $h^0(\RR, \RR) = h^0(\GG_1, \GG_1) + h^0(\GG_2, \GG_2) +\chi(\GG_1, \GG_2)$;
 \item $h^2(\RR, \RR) = h^2(\GG_1, \GG_1) + h^2(\GG_2, \GG_2) +\chi(\GG_2, \GG_1)$;
\item $h^1(\GG_1, \GG_2)\leq h^0(\GG_1, \GG_1) +h^0(\GG_2, \GG_2)-1$.
\end{enumerate}
\end{lem}
\begin{proof}
(1)--(3) are from \cite[2.2 Lemma]{KO}. We prove (4) here. In the proof of \cite[2.2 Lemma]{KO}, we know that the natural map
$$
\Hom(\GG_1, \GG_1) \oplus \Hom(\GG_2, \GG_2) \xrightarrow{d_1} \Ext^1(\GG_1, \GG_2)
$$
is surjective. Note that the image of $(\id_{\GG_1}, \id_{\GG_2})$ is zero. Hence we get the inequality by comparing the dimensions.
\end{proof}

\begin{lem} \label{lem EES} Consider an exact sequence of coherent sheaves
$$0 \to \EE'  \to \EE \to \SSS \to 0,$$
where $\EE$ is rigid, $\SSS$ is spherical, $h^0(\EE', \SSS) =0$, and $\chi(\SSS, \EE')=-1$. 
Then $h^i(\EE', \EE')=h^i(\EE, \EE)$ for $i=0,1,2$. In particular, $\EE$ is exceptional if and only if so is $\EE'$, and in this case, $\EE\cong T_\SSS( \EE')$.
\end{lem}

\begin{proof}
Since $\SSS$ is spherical, $h^2(\SSS, \EE') =0$ and $\chi(\EE', \SSS)=-1$ by Serre duality. By Lemma \ref{lem GEG}, $h^i(\EE', \EE')=h^i(\EE, \EE)$ for $i=0,1,2$. In particular, $\EE$ is exceptional if and only if so is $\EE'$. 

Suppose that $\EE$ and $\EE'$ are exceptional, then by Lemma \ref{lem GEG}(4), 
$$h^1(\SSS,\EE')\leq h^0(\SSS, \SSS) +h^0(\EE', \EE')-1=1.$$
Since $\chi(\SSS,\EE')=-1$, we have $h^1(\SSS,\EE')=1$ and $h^0(\SSS, \EE')=h^2(\SSS, \EE') =0$.
By definition of twist functor, we have a distinguished triangle
$$
\SSS[-1] \to \EE' \to T_{\SSS}(\EE'),
$$
which corresponds to the exact sequence
$$0 \to \EE'  \to \EE \to \SSS \to 0.$$
Hence  $\EE\cong T_\SSS (\EE')$.
\end{proof}


We can say more about torsion rigid sheaves. 
\begin{remark}\label{remark restriction}
Let $\RR$ be a torsion rigid sheaf, then $\RR$ is pure one-dimensional by \cite[Corollary 2.2.3]{Kuleshov}. Suppose that  
$\supp(\RR)=Z\cup Z'$ where $Z$ and $Z'$ are unions of curves with no common components. Then there exists an exact sequence
$$
0\to \RR' \to \RR\to \RR_{Z} \to 0
$$
where $\RR'=\HH^0_{Z'}(\RR)$ is the subsheaf with supports (see \cite[II, Ex. 1.20]{H}) in $Z'$ and $\RR_{Z}$ is the quotient sheaf.
Then $\supp(\RR')=Z'$,  $\supp(\RR_{Z})=Z$, and $h^0(\RR', \RR_{Z})=h^2(\RR_{Z}, \RR')=0$ by the support condition. By Lemma \ref{lem GEG},  $\RR'$ and $\RR_{Z}$ are torsion rigid (pure one-dimensional) sheaves. Moreover, if we write $Z=\cup_{i}C_i$ and $Z'=\cup_j C'_j$, then we can write the first Chern class of $\RR$ uniquely as $$c_1(\RR)=\sum_ir_iC_i+\sum_j s_jC'_j$$
in the sense that $c_1(\RR_{Z})=\sum_ir_iC_i$ and $c_1(\RR')=\sum_j s_jC'_j$ for some positive integers $r_i$ and $s_j$. 
\end{remark}

\begin{definition}[restriction to curves]\label{def restriction}
Let $\RR$ be a torsion rigid sheaf such that
$\supp(\RR)=Z\cup Z'$ where $Z$ and $Z'$ are unions of curves with no common components. 
The sheaf $\RR_{Z}$ constructed as in Remark \ref{remark restriction} is called {\it the restriction of $\RR$ to $Z$}.
\end{definition}

Note that for an irreducible curve $C\subset \supp(\RR)$, the restriction $\RR_C$ of $\RR$ to $C$ is sometimes different from the restriction $\RR|_C:=\RR \otimes \OO_C$. For example, if $\RR=\OO_{2C}$, then $\RR_C=\OO_{2C}$ while $\RR|_C=\OO_C$.

Here we remark that for a chain of $(-2)$-curves $Z=C_1\cup\cdots\cup C_n$ and a torsion sheaf $\RR$, $\supp(\RR)\subset Z$ if and only if $c_1({\RR})$ can be written as $\sum_ir_iC_i$ for non-negative integers $r_i$. In this case, $r_i$ is uniquely determined for all $i$.

\begin{lem}\label{lem supp}
Let $\RR$ be a torsion rigid sheaf on $X$, then any irreducible component of $\supp(\RR)$ is a curve with negative self-intersection.
\end{lem}

\begin{proof}
Let $C$ be an irreducible component of $\supp(\RR)$. Take the restriction to $C$, we have an exact sequence 
$$0\to  \RR'\to \RR\to \RR_C\to 0. $$ 
By Remark \ref{remark restriction}, $\RR_C$ is rigid and in particular,
$\chi(\RR_C,\RR_C)>0$. On the other hand, by Riemann--Roch formula (see Subsection \ref{subsection RR}), we have
$\chi(\RR_C,\RR_C)=-c_1(\RR_C)^2.$ Hence $c_1(\RR_C)^2<0$, which implies that $C^2<0$.
\end{proof}

\subsection{Weak del Pezzo surfaces}
A smooth projective surface $X$ is a  {\it weak del Pezzo surface} if $-K_X$ is nef and big. The {\it degree} of $X$ is the self-intersection number $(-K_X)^2$. A weak del Pezzo surface is of {\it Type A} if its anticanonical model has at most $A_n$-singularities. We collect some basic facts on weak del Pezzo surfaces.

\begin{lem}[{cf. \cite[Theorem 8.3.2]{Dolgachev}}]\label{lem no base}
Let $X$ be a weak del Pezzo surface. Then $|-K_X|$ has no base components.
\end{lem}

\begin{lem}[{cf. \cite[Theorem 8.2.25]{Dolgachev}}]\label{lem no of -2}
Let $X$ be a weak del Pezzo surface of degree $d$. Then the number of $(-2)$-curves on $X$ is at most $9-d$.
\end{lem}

\begin{lem}\label{lem d>1}
Let $X$ be a weak del Pezzo surface of degree $d>1$.
Then the intersection number of a $(-1)$-curve with a chain of $(-2)$-curves is at most one.
\end{lem}
\begin{proof}
Take a chain of $(-2)$-curves $C_1\cup\cdots\cup C_n$ and a $(-1)$-curve $D$. Assume that $\sum_{i=1}^nC_i\cdot D\geq 2$, then 
$$\Big(\sum_{i=1}^nC_i+D\Big)^2=\Big(\sum_{i=1}^nC_i\Big)^2+2\sum_{i=1}^nC_i\cdot D+D^2=-2+2\sum_{i=1}^nC_i\cdot D-1\geq 1.$$
By the Hodge index theorem,
$$
(-K_X)^2\cdot\Big(\sum_{i=1}^nC_i+D\Big)^2\leq \Big((-K_X)\cdot\Big(\sum_{i=1}^nC_i+D\Big)\Big)^2=1.
$$
This implies that $(-K_X)^2=1$, which is a contradiction.
\end{proof}
We remark that on weak del Pezzo surfaces of degree one, it is possible that one $(-1)$-curve intersects with a chain of $(-2)$-curves at two points (cf. \cite[Lemma 2.8]{Kosta}).
\subsection{Riemann--Roch formula}\label{subsection RR}
We recall Riemann--Roch formula for torsion sheaves on surfaces (cf. \cite[(1.1)]{KO}).
\begin{lem}[Riemann--Roch formula for torsion sheaves]
For two torsion sheaves $\EE$ and $\FF$ on a smooth projective surface $X$, the Euler characteristic can be calculated by
$
\chi(\EE, \FF)=-c_1(\EE)\cdot c_1(\FF).$
\end{lem}

\subsection{A polynomial inequality}\label{polynomial}
In this subsection, we treat a special polynomial which naturally appears in self-intersection numbers of a union of negative curves.

For positive integers $r_1, r_2, \ldots, r_n$, and $1\leq k\leq n$, define the polynomial
$$f(r_1, r_2, \ldots, r_n; k)=\sum_{i=1}^n r_i^2-\sum_{i=1}^{n-1}r_ir_{i+1}-r_k.$$
\begin{prop}\label{prop polynomial}
For positive integers $r_1, r_2, \ldots, r_n$, and $k$, 
$$f(r_1, r_2, \ldots, r_n; k)\geq 0$$
always holds. Moreover, $f(r_1, r_2, \ldots, r_n; k)= 0$ if and only if the following conditions hold:
\begin{enumerate}
\item $r_1=r_n=1$,
\item $0\leq r_{i+1}-r_i\leq 1$ if $i<k$,
\item $0\leq r_{i}-r_{i+1}\leq 1$ if $i\geq k$.
\end{enumerate}
\end{prop}
\begin{proof}
It is easy to see that
\begin{align*}
 {}&2f(r_1, r_2, \ldots, r_n; k)\\
={}&r^2_1+\sum_{i=1}^{n-1}(r_i-r_{i+1})^2+r^2_n-2r_k\\
\geq {}&r_1+\sum_{i=1}^{n-1}|r_i-r_{i+1}|+r_n-2r_k\\
\geq {}&\left(r_1+\sum_{i=1}^{k-1}(r_{i+1}-r_i)\right)+\left(\sum_{i=k}^{n-1}(r_i-r_{i+1})+r_n\right)-2r_k=0.
\end{align*}
This proves the inequality $f(r_1, r_2, \ldots, r_n; k)\geq 0$. If the equality holds, then all the above inequalities become equalities.
From the first inequality, we get $r_1=r_n=1$ and $|r_i-r_{i+1}|\leq 1$ for $1\leq i\leq n-1$. From the second, we get $r_{i+1}-r_i\geq  0$ if $i<k$, and $ r_{i}-r_{i+1}\geq 0$ if $i\geq k$.
 \end{proof}

\section{Examples}\label{section example}
\noindent In this section, before proving the theorems, we provide several interesting examples of torsion exceptional sheaves on weak del Pezzo surfaces.
The examples illuminate how we may apply  Lemma \ref{lem EES} to reduce the torsion exceptional sheaf to a line bundle on a $(-1)$-curve by spherical twists.

\begin{example}
Let $X$ be a weak del Pezzo surface of degree $d>1$ whose anticanonical model has at most $A_1$-singularities. Then by our proof, every torsion exceptional sheaf on $X$ has the form $\OO_{D\cup C_1\cup\cdots \cup C_n}(d, a_1, \ldots, a_n)$, where $C_i$ are disjoint $(-2)$-curves, $D$ is a $(-1)$-curve intersecting with each $C_i$, and $d, a_i$ are integers. Note that $n$ can be $0$ which means there is no $(-2)$-curve in the support. Similar result holds true for weak del Pezzo surfaces of degree $d>1$ whose anticanonical model has at most $A_2$-singularities.
\end{example}

The following example shows that the scheme theoretic support of a  torsion exceptional sheaf can be non-reduced.
\begin{example}\label{example 2}Let $X$ be a smooth projective surface, $C_1\cup C_2\cup C_3$ a chain of three $(-2)$-curves, and $D$ a $(-1)$-curve. Assume that $D\cdot C_2=1$ and $D\cdot C_1=D\cdot C_3=0$. Then the structure sheaf $\OO_{D\cup C_1\cup 2C_2\cup C_3}$ is a torsion exceptional sheaf on $X$ with non-reduced support. In fact, by applying Lemma \ref{lem EES} to the following exact sequences
$$
0\to \OO_{D\cup C_1\cup C_2\cup C_3}(-1, -1, 2, -1) \to  \OO_{D\cup C_1\cup 2C_2\cup C_3} \to  \OO_{C_2} \to 0,
$$
$$
0\to \OO_{D}(-2) \to  \OO_{D\cup C_1\cup C_2\cup C_3}(-1, -1, 2, -1) \to  \OO_{C_1\cup C_2\cup C_3}( -1, 2, -1)  \to 0,
$$
we get
$$
\OO_{D\cup C_1\cup 2C_2\cup C_3}=T_{\OO_{C_2}}\circ T_{\OO_{C_1\cup C_2\cup C_3}( -1, 2, -1)}( \OO_{D}(-2) ).
$$
\end{example}
The following example shows that the support of a  torsion exceptional sheaf on a weak del Pezzo surface of degree one can contain loops.
\begin{example}
Let $X$ be a smooth projective surface, $C_1\cup C_2\cup C_3$ a chain of three $(-2)$-curves, and $D$ a $(-1)$-curve. Assume that $D\cdot C_2=0$ and $D\cdot C_1=D\cdot C_3=1$, that is, $C_1, C_2, C_3, D$ form a loop (this might happen, for example, on the minimal resolution of a singular del Pezzo surface of degree one with one $A_3$-singularity, cf. \cite[Lemma 2.8]{Kosta}). Then the unique non-trivial extension $\EE$ of $\OO_{C_2\cup C_3}$ by $\OO_{D\cup C_1\cup C_2}$ is a torsion exceptional sheaf on $X$ whose support is a loop. In fact, 
$$
h^1(\OO_{C_2\cup C_3}, \OO_{D\cup C_1\cup C_2})=1
$$
since 
\begin{align*}
h^0(\OO_{C_2\cup C_3}, \OO_{D\cup C_1\cup C_2}){}&=0,\\
\chi(\OO_{C_2\cup C_3}, \OO_{D\cup C_1\cup C_2}){}&=-1,\\
h^0(\OO_{D\cup C_1\cup C_2}, \OO_{C_2\cup C_3}){}&=0.
\end{align*}
By applying Lemma \ref{lem EES} twice, we get
$$
\EE=T_{\OO_{C_2\cup C_3}}\circ T_{\OO_{C_1\cup C_2}}( \OO_{D}(-1) ).
$$
\end{example}

From these examples, one can get a rough idea on how we apply  Lemma \ref{lem EES} to cut down the support of a torsion exceptional sheaf. In fact, it should be easy to verify Conjecture \ref{conj} for the structure sheaf $\OO_\Gamma$ on arbitrary weak del Pezzo surface where $\Gamma$ is a proper subscheme and $\OO_\Gamma$ is assumed to be a torsion exceptional sheaf. However, in general torsion exceptional sheaf could not have such a good structure.
\section{Factorizations of rigid sheaves}
\noindent In this section, we assume that $X$ is a  smooth projective surface. All sheaves are considered to be coherent on $X.$ We will define factorizations of rigid sheaves and give basic properties.
\begin{definition}
A coherent sheaf $\RR$ has a {\it factorization}
$
(\GG_1, \GG_2,\ldots, \GG_n)
$
if there exists a filtration of coherent sheaves
$$
0= \FF_0\subset \FF_1\subset \FF_2 \subset \cdots \subset \FF_n=\RR,
$$
such that
$\FF_i/\FF_{i-1}\cong \GG_i$ for $1\leq i\leq n$ 
and we write this factorization as
$$\RR\equiv (\GG_1, \ldots, \GG_n).$$ 
This factorization is said to be {\it perfect} if $h^0(\GG_i, \GG_j)=h^2(\GG_j, \GG_i)=0$ for all $i<j$.
\end{definition}
\begin{remark}For a factorization $\RR\equiv (\GG_1, \ldots, \GG_n)$ and some index $1\leq i_0\leq n$, if $\GG_i\cong \GG_i\otimes \omega_X$ for all $i\neq i_0$ (for example, $\GG_i$ is a torsion sheaf whose support consists of $(-2)$-curves), then $h^0(\GG_i, \GG_j)=h^2(\GG_j, \GG_i)$ for all $i<j$ by duality. Hence to check the perfectness of this factorization, one only need to check that $h^0(\GG_i, \GG_j)=0$ for all $i<j$. Note that in subsequent sections, we are always in such situation when we deal with factorizations.
\end{remark}
\begin{example}[{\cite[Lemma 2.4]{OU}}]\label{example HN filtration}
Let $C$ be a $(-2)$-curve on a smooth projective surface $X$.
Let $\FF$ be a pure one-dimensional sheaf on the scheme $mC$.
Then the subquotients of the Harder--Narasimhan filtration
$$
0=\FF^0\subset\FF^1 \subset\cdots  \subset\FF^n =\FF$$
 of $\FF$ are of the form
$$
(\FF^1 / \FF^0,
\FF^2 / \FF^1,
\dots,
\FF^n / \FF^{n-1})
=
(\OO_C(a_1)^{\oplus r_1}, \OO_C(a_2)^{\oplus r_2}, \dots, \OO_C(a_n)^{\oplus r_n})
$$
with $a_1>a_2>\cdots>a_n$ and $r_i>0$, which gives a perfect factorization of $\FF$.
\end{example}

The following lemma is a direct consequence of Lemma \ref{lem GEG}.

\begin{lem}\label{lem G rigid}
Let $\RR$ be a rigid sheaf with a perfect factorization
$
(\GG_1,\ldots, \GG_n)
$
and $\FF_0\subset \FF_1\subset \FF_2 \subset \cdots \subset \FF_n
$ the corresponding filtration.
Then $\FF_j/\FF_i$ is rigid for all $i<j$.
\end{lem}

In applications, we usually need to get new factorizations from old ones.
Here we give some lemmas about operations on  factorizations.

\begin{lem}\label{lem change}
Let $\RR$ be a coherent sheaf with a factorization
$
(\GG_1,\ldots, \GG_n)
$
and $\FF_0\subset \FF_1\subset \FF_2 \subset \cdots \subset \FF_n
$ the corresponding filtration.
If $\FF_{i+1}/\FF_{i-1}$ has another factorization $(\GG'_{i}, \GG'_{i+1})$ for some $i$, 
then  $\RR$ has another factorization
$$
(\GG_1,\ldots, \GG_{i-1}, \GG'_i,\GG'_{i+1}, \GG_{i+2}, \ldots, \GG_n).
$$
In particular, if $h^1(\GG_{i+1}, \GG_i)=0$, then we are free to change the order of $\GG_{i+1}$ and  $\GG_i$ in a  factorization.
\end{lem}
\begin{proof}Easy.
\end{proof}
\begin{lem}
Let $\RR$ be a coherent sheaf with a perfect factorization
$
(\GG_1,\ldots, \GG_n)
$
and $\FF_0\subset \FF_1\subset \FF_2 \subset \cdots \subset \FF_n
$ the corresponding filtration. Then 
$
(\FF_{n-1}, \GG_n)
$
is a perfect factorization of $\RR$.
\end{lem}
\begin{proof}Easy.
\end{proof}
\begin{lem}\label{lem extend GS}
Let $\GG$ and $\SSS$ be two coherent sheaves and $r$ a positive integer. Assume that $\SSS$ is simple (i.e., $h^0(\SSS,\SSS)=1$). 
\begin{enumerate}
\item Assume that $h^1 (\SSS, \GG)=1$. Denote by $\GG'$ the unique non-trivial extension of $\SSS$ by $\GG$. 
Then $h^0(\SSS, \GG)=h^0(\SSS, \GG')$, and any non-trivial extension of $\SSS^{\oplus r}$ by $\GG$ is isomorphic to $\SSS^{\oplus (r-1)}\oplus \GG'$. 

\item Assume  that $h^1 (\GG, \SSS)=1$. Denote by $\GG''$ the unique non-trivial extension of $\GG$ by $\SSS$. 
Then $h^0(\GG, \SSS)=h^0(\GG'', \SSS)$, and any non-trivial extension of $\GG$ by $\SSS^{\oplus r}$ is isomorphic to $\SSS^{\oplus (r-1)}\oplus \GG''$. 

\end{enumerate}
\end{lem}
\begin{proof}
(1) Consider the exact sequence
$$
0\to  \Hom(\SSS, \GG) \to \Hom(\SSS, \GG') \to \Hom(\SSS, \SSS) \overset{\delta}{\to} \Ext^1(\SSS, \GG)
$$
induced by the extension $$0\to \GG\to \GG'\to \SSS\to 0.$$ Because the extension is non-trivial and $\SSS$ is simple, the map $\delta$ is injective. Hence $h^0(\SSS, \GG)=h^0(\SSS, \GG')$.

Consider a non-trivial extension $\RR$ corresponding to 
$$
\eta=(\eta_1, \ldots, \eta_r)\in \Ext^1(\SSS^{\oplus r}, \GG)\cong \CC^r.
$$
As $\SSS$ is simple, $\Aut(\SSS^{\oplus r})=\GL(r, \CC)$. Since $\Aut(\SSS^{\oplus r})$ acts on $\Ext^1(\SSS^{\oplus r}, \GG)\cong \CC^r$
through the natural action of $\GL(r, \CC)$, after taking an automorphism of $\SSS^{\oplus r}$, we may assume that $\eta_i=0$ except for one index $i_0$. Hence $\RR\cong \SSS^{\oplus (r-1)}\oplus \GG'$.

(2) can be proved similarly.
\end{proof}

\begin{lem}\label{lem GSH}
Let $\RR$ be a rigid sheaf with a perfect factorization
$$
(\GG_1,\ldots, \GG_n, \SSS^{\oplus r}, \HH_1,\ldots, \HH_m).
$$
Assume that $\SSS$ is spherical.
\begin{enumerate}
 \item Suppose that $h^0 (\SSS, \GG_n)=0$ and $\chi  (\SSS, \GG_n)=-1$, then
there is a new perfect factorization
$$
(\GG_1,\ldots, \GG_{n-1}, \SSS^{\oplus (r-1)}, \GG'_n, \HH_1,\ldots, \HH_m).
$$
Here $\GG'_n$ is the (unique) non-trivial extension of $\SSS$ by $\GG_n$. 

 \item Suppose that $h^0 (\HH_1, \SSS)=0$ and  $\chi   (\HH_1, \SSS)=-1$, then
there is a new perfect factorization
$$
(\GG_1,\ldots, \GG_n, \HH'_1,  \SSS^{\oplus (r-1)}, \HH_2,\ldots, \HH_m).
$$
Here $\HH'_1$ is the (unique) non-trivial extension of $\HH_1$ by $\SSS$. 

\end{enumerate}
\end{lem}

\begin{proof}(1)
Since the factorization is perfect, $h^0 (\GG_n, \SSS)=h^2 (\SSS, \GG_n)=0$. Hence $\chi  (\SSS, \GG_n)=-1$ implies that $h^1 (\SSS, \GG_n)=1$. The unique non-trivial extension $\GG'_n$ of $\SSS$ by $\GG_n$ is well-defined. 
Note that the perfect factorization
$$
(\GG_1,\ldots, \GG_n, \SSS^{\oplus r}, \HH_1,\ldots, \HH_m).
$$
induces another 
perfect factorization
$$
(\GG_1,\ldots, \GG_{n-1}, \FF', \HH_1,\ldots, \HH_m).
$$
where $\FF'$ is an extension of $\SSS^{\oplus r}$ by $\GG_n$. By Lemma \ref{lem G rigid}, $\FF'$ is rigid, hence the extension is non-trivial. By Lemma \ref{lem extend GS}(1), $\FF'\cong \SSS^{\oplus (r-1)}\oplus \GG'_{n}$. Hence there exists a factorization
$$
(\GG_1,\ldots, \GG_{n-1}, \SSS^{\oplus r},\GG'_{n}, \HH_1,\ldots, \HH_m).
$$
It is easy to check that this factorization is perfect, since $h^0(\SSS,\GG'_{n})=h^0(\SSS,\GG_{n})=0$ by Lemma \ref{lem extend GS}(1) and $h^2(\GG'_{n},\SSS)=h^0(\SSS,\GG'_{n})=0$ by duality.

(2) can be proved similarly.
\end{proof}

\section{Torsion rigid sheaves supported on $(-2)$-curves}\label{section rigid -2}
\noindent In this section, we assume that $X$ is a smooth projective surface. All sheaves are considered to be coherent on $X.$ We will classify certain torsion rigid sheaves 
supported on $(-2)$-curves.
\begin{prop}\label{prop 12}
Let $C_1\cup C_2$ be a chain of two $(-2)$-curves and $\RR$ a torsion rigid sheaf with $c_1(\RR)=C_1+2C_2$. Then $\RR$ has one of the following perfect factorizations:
\begin{enumerate}
\item $(\OO_{C_2}(a_2), \OO_{C_1\cup C_2}(a_1, a_2));$
\item $( \OO_{C_1\cup C_2}(a_1, a_2), \OO_{C_2}(a_2-1));$
\item$ ( \OO_{C_1\cup C_2}(a_1, a_2), \OO_{C_2}(a_2-2)).$
\end{enumerate}
Here $a_1, a_2$ are integers.
\end{prop}
\begin{proof}
Taking the restriction to $C_2$, we have an exact sequence
$$
0\to \RR_1\to \RR \to \RR_2\to 0.
$$
By Remark \ref{remark restriction}, $\RR_1$ and $\RR_2$ are rigid. Note that $c_1(\RR_1)=C_1$ and $c_1(\RR_2)=2C_2$. Hence $\RR_1=\OO_{C_1}(a_1-1)$ is a line bundle on $C_1$ for some integer $a_1$, and $\RR_2$ has a perfect factorization induced by Harder--Narasimhan filtration, which is 

{\bf Case 1.} $(\OO_{C_2}(a_2)^{\oplus 2})$ for some integer $a_2$, or

{\bf Case 2.} $(\OO_{C_2}(a_2), \OO_{C_2}(b_2))$, for integers $a_2>b_2$.

In Case 1, $\RR$ has a perfect factorization
$$(\OO_{C_1}(a_1-1), \OO_{C_2}(a_2)^{\oplus 2}),$$
for which we can apply Lemma \ref{lem GSH} to get a new  perfect factorization
$$(\OO_{C_2}(a_2), \OO_{C_1\cup C_2}(a_1, a_2)).$$
This gives (1).

In Case 2, since $\RR_2$ is rigid, by Lemma \ref{lem GEG}(4), we have
$$h^1( \OO_{C_2}(b_2),\OO_{C_2}(a_2))\leq 1,$$
which implies that $a_2\leq b_2+2$, that is, $b_2=a_2-1$ or $a_2-2$.
In this case,
$\RR$ has a perfect factorization
$$(\OO_{C_1}(a_1-1), \OO_{C_2}(a_2),\OO_{C_2}(b_2)),$$
for which we can apply Lemma \ref{lem GSH} to get a new  perfect factorization
$$( \OO_{C_1\cup C_2}(a_1, a_2),\OO_{C_2}(b_2)).$$
This gives (2) and (3).
\end{proof}

\begin{prop}\label{prop 123}
Let $C_1\cup C_2\cup C_3$ be a chain of three $(-2)$-curves and $\RR$ a torsion rigid sheaf with $c_1(\RR)=C_1+2C_2+3C_3$. Then $\RR$ has one of the following perfect factorizations:
\begin{enumerate}
\item[(1-1)] $(\OO_{C_3}(a_3),\OO_{C_2\cup C_3}(a_2, a_3), \OO_{C_1\cup C_2 \cup C_3}(a_1, a_2, a_3));$
\item[(1-2)] $(\OO_{C_3}(a_3),\OO_{C_1\cup C_2 \cup C_3}(a_1, a_2, a_3), \OO_{C_2\cup C_3}(a_2-1, a_3));$
\item[(1-3)] $(\OO_{C_3}(a_3),\OO_{C_1\cup C_2 \cup C_3}(a_1, a_2, a_3), \OO_{C_2\cup C_3}(a_2-2, a_3));$

\item[(2-1)] $(\OO_{C_2\cup C_3}(a_2, a_3), \OO_{C_1\cup C_2 \cup C_3}(a_1, a_2, a_3), \OO_{C_3}(b_3));$
\item[(2-2)] $(\OO_{C_1\cup C_2 \cup C_3}(a_1, a_2, a_3), \OO_{C_2\cup C_3}(a_2-1, a_3),\OO_{C_3}(b_3));$
\item[(2-3)] $(\OO_{C_1\cup C_2 \cup C_3}(a_1, a_2, a_3), \OO_{C_2\cup C_3}(a_2-2, a_3),\OO_{C_3}(b_3));$

\item[(3-1)] $(\OO_{C_1\cup C_2 \cup C_3}(a_1, a_2, a_3),\OO_{C_3}(b_3), \OO_{C_2\cup C_3}(a_2, b_3));$
\item[(3-2)] $(\OO_{C_2\cup C_3}(a_2, a_3),\OO_{C_3}(b_3), \OO_{C_1\cup C_2 \cup C_3}(a_1, a_2+1, b_3));$
\item[(3-3)] $(\OO_{C_1\cup C_2 \cup C_3}(a_1, a_2, a_3),\OO_{C_3}(b_3), \OO_{C_2\cup C_3}(a_2-1, b_3));$
\item[(3-4)] $(\OO_{C_1\cup 2C_2 \cup C_3}(a_1, a_2, a_3)\oplus \OO_{C_3}(a_3-1)^{\oplus 2});$

\item[(4-1)] $(\OO_{C_1\cup C_2 \cup C_3}(a_1, a_2, a_3), \OO_{C_2\cup C_3}(a_2, b_3),\OO_{C_3}(c_3));$

\item[(4-2)] $(\OO_{C_2\cup C_3}(a_2, a_3), \OO_{C_1\cup C_2 \cup C_3}(a_1, a_2+1, b_3), \OO_{C_3}(c_3));$
\item[(4-3)] $(\OO_{C_1\cup C_2 \cup C_3}(a_1, a_2, b_3+1), \OO_{C_2\cup C_3}(a_2-1, b_3),\OO_{C_3}(c_3));$
\item[(4-4)] $(\OO_{C_1\cup 2C_2 \cup C_3}(a_1, a_2, b_3+1)\oplus \OO_{C_3}(b_3), \OO_{C_3}(c_3)).$
\end{enumerate}
Here $a_i, b_i, c_i$ are integers and $a_3>b_3>c_3$.
\end{prop}
\begin{proof}
Taking the restriction to $C_3$, we have an exact sequence
$$
0\to \RR_{12}\to \RR \to \RR_3\to 0.
$$
By Remark \ref{remark restriction}, $\RR_{12}$ and $\RR_{3}$ are rigid. Note that $c_1(\RR_{12})=C_1+2C_2$ and $c_1(\RR_3)=3C_3$. $\RR_3$ has a perfect factorization induced by Harder--Narasimhan filtration, we have 4 cases: 

{\bf Case 1.} $(\OO_{C_3}(a_3)^{\oplus 3})$ for some integer $a_3$;

{\bf Case 2.} $(\OO_{C_3}(a_3)^{\oplus 2}, \OO_{C_3}(b_3))$, for integers $a_3>b_3$;

{\bf Case 3.} $(\OO_{C_3}(a_3), \OO_{C_3}(b_3)^{\oplus 2})$, for integers $a_3>b_3$;

{\bf Case 4.} $(\OO_{C_3}(a_3), \OO_{C_3}(b_3),\OO_{C_3}(c_3))$, for integers $a_3>b_3>c_3$.

Each case can be divided in to 3 subcases according to the perfect factorization $\RR_{12}\equiv (\GG_1, \GG_2)$ in Proposition \ref{prop 12}.

In Case 1, applying Lemma \ref{lem GSH} twice with $\SSS=\OO_{C_3}(a_3)$ to the perfect factorization $\RR\equiv (\GG_1, \GG_2, \OO_{C_3}(a_3)^{\oplus 3}),$ we get a new perfect factorization, which gives (1-1), (1-2), or (1-3) by changing $a_2$ appropriately.

In Case 2, applying Lemma \ref{lem GSH} twice with $\SSS=\OO_{C_3}(a_3)$  to the perfect factorization $\RR\equiv (\GG_1, \GG_2, \OO_{C_3}(a_3)^{\oplus 2}, \OO_{C_3}(b_3)),$ we get a new perfect factorization, which gives (2-1), (2-2), or (2-3) by changing $a_2$ appropriately.

In Case 3, we have 3 subcases:

{\bf Subcase 3.1.} $\RR_{12}\equiv (\OO_{C_2}(a_2), \OO_{C_1\cup C_2}(a_1, a_2));$

{\bf Subcase 3.2.} $\RR_{12}\equiv ( \OO_{C_1\cup C_2}(a_1, a_2), \OO_{C_2}(a_2-1));$

{\bf Subcase 3.3.} $\RR_{12}\equiv ( \OO_{C_1\cup C_2}(a_1, a_2), \OO_{C_2}(a_2-2)).$

In Subcase 3.1, $\RR$ has a perfect factorization 
$$(\OO_{C_2}(a_2), \OO_{C_1\cup C_2}(a_1, a_2), \OO_{C_3}(a_3), \OO_{C_3}(b_3)^{\oplus 2}).$$
Applying Lemma \ref{lem GSH}, we get a new perfect factorization
$$(\OO_{C_2}(a_2), \OO_{C_1\cup C_2\cup C_3 }(a_1, a_2+1, a_3), \OO_{C_3}(b_3)^{\oplus 2}).$$
Note that Hom's and $\chi$ between the first two factors are trivial, we get 
$$
h^1( \OO_{C_1\cup C_2\cup C_3 }(a_1, a_2+1, a_3), \OO_{C_2}(a_2))=0,
$$ 
and we can exchange the first two factors to get a new perfect factorization
$$(\OO_{C_1\cup C_2\cup C_3 }(a_1, a_2+1, a_3), \OO_{C_2}(a_2), \OO_{C_3}(b_3)^{\oplus 2}).$$
Applying Lemma \ref{lem GSH}(1), we get a new perfect factorization
$$(\OO_{C_1\cup C_2\cup C_3 }(a_1, a_2+1, a_3), \OO_{C_3}(b_3),  \OO_{C_2\cup C_3}(a_2+1, b_3)).$$
This gives (3-1) by changing $a_2$ appropriately.

In Subcase 3.2, $\RR$ has a perfect factorization 
$$( \OO_{C_1\cup C_2}(a_1, a_2), \OO_{C_2}(a_2-1), \OO_{C_3}(a_3), \OO_{C_3}(b_3)^{\oplus 2}).$$
Applying Lemma \ref{lem GSH}, we get a new perfect factorization
$$( \OO_{C_1\cup C_2}(a_1, a_2), \OO_{C_2\cup C_3}(a_2, a_3), \OO_{C_3}(b_3)^{\oplus 2}).$$
Note that Hom's and $\chi$ between the first two factors are trivial, we get 
$$
h^1(\OO_{C_2\cup C_3}(a_2, a_3), \OO_{C_1\cup C_2}(a_1, a_2))=0,
$$ 
and we can exchange the first two factors to get a new perfect factorization
$$(\OO_{C_2\cup C_3}(a_2, a_3), \OO_{C_1\cup C_2}(a_1, a_2), \OO_{C_3}(b_3)^{\oplus 2}).$$
Applying Lemma \ref{lem GSH}(1), we get a new perfect factorization
$$(\OO_{C_2\cup C_3}(a_2, a_3), \OO_{C_3}(b_3),  \OO_{C_1\cup C_2\cup C_3}(a_1, a_2+1, b_3)).$$
This gives (3-2).

In Subcase 3.3, $\RR$ has a perfect factorization 
$$( \OO_{C_1\cup C_2}(a_1, a_2), \OO_{C_2}(a_2-2), \OO_{C_3}(a_3), \OO_{C_3}(b_3)^{\oplus 2}).$$
Applying Lemma \ref{lem GSH}, we get a new perfect factorization
$$( \OO_{C_1\cup C_2}(a_1, a_2), \OO_{C_2\cup C_3}(a_2-1, a_3), \OO_{C_3}(b_3)^{\oplus 2}).$$
Note that  
$$h^1(\OO_{C_2\cup C_3}(a_2-1, a_3), \OO_{C_1\cup C_2}(a_1, a_2))=1$$
since
\begin{align*}
\chi ( \OO_{C_2\cup C_3}(a_2-1, a_3), \OO_{C_1\cup C_2}(a_1, a_2)){}&=0,\\
h^0 ( \OO_{C_2\cup C_3}(a_2-1, a_3), \OO_{C_1\cup C_2}(a_1, a_2)){}&=1,\\
h^0( \OO_{C_1\cup C_2}(a_1, a_2), \OO_{C_2\cup C_3}(a_2-1, a_3)){}&=0,
\end{align*}
and the unique non-trivial extension is $\OO_{C_1\cup 2C_2\cup C_3}(a_1+1, a_2-1, a_3)$, we get a new perfect factorization
$$( \OO_{C_1\cup 2C_2\cup C_3}(a_1+1, a_2-1, a_3), \OO_{C_3}(b_3)^{\oplus 2}).$$
Note that $\OO_{C_1\cup 2C_2\cup C_3}(a_1+1, a_2-1, a_3)$ can be also viewed as the extension of $\OO_{C_2}(a_2-1)$ by $\OO_{C_1\cup C_2\cup C_3}(a_1, a_2+1, a_3-1)$.

Now if $a_3>b_3+1$, then $\RR$ has a new factorization
$$
( \OO_{C_1\cup C_2\cup C_3}(a_1, a_2+1, a_3-1), \OO_{C_2}(a_2-1), \OO_{C_3}(b_3)^{\oplus 2}),
$$
which is perfect by checking Hom's. Applying Lemma \ref{lem GSH}, we get a new perfect factorization
$$
( \OO_{C_1\cup C_2\cup C_3}(a_1, a_2+1, a_3-1), \OO_{C_3}(b_3), \OO_{C_2\cup C_3}(a_2, b_3)).
$$
This gives (3-3) by changing $a_2, a_3$ appropriately.

If $a_3=b_3+1$, then 
$$
h^1( \OO_{C_3}(b_3),\OO_{C_1\cup 2C_2\cup C_3}(a_1+1, a_2-1, a_3))=0
$$
since 
\begin{align*}
\chi ( \OO_{C_3}(b_3),\OO_{C_1\cup 2C_2\cup C_3}(a_1+1, a_2-1, a_3)){}&=0,\\
h^0 ( \OO_{C_3}(b_3),\OO_{C_1\cup 2C_2\cup C_3}(a_1+1, a_2-1, a_3)){}&=0,\\
h^0( \OO_{C_1\cup 2C_2\cup C_3}(a_1+1, a_2-1, a_3), \OO_{C_3}(b_3)){}&=0.
\end{align*}
Hence 
$$
\RR\cong \OO_{C_1\cup 2C_2 \cup C_3}(a_1+1, a_2-1, a_3)\oplus \OO_{C_3}(a_3-1)^{\oplus 2},
$$
which gives (3-4) by changing $a_1, a_2$ appropriately.

Finally we consider Case 4. Again we have 3 subcases:

{\bf Subcase 4.1.} $\RR_{12}\equiv (\OO_{C_2}(a_2), \OO_{C_1\cup C_2}(a_1, a_2));$

{\bf Subcase 4.2.} $\RR_{12}\equiv ( \OO_{C_1\cup C_2}(a_1, a_2), \OO_{C_2}(a_2-1));$

{\bf Subcase 4.3.} $\RR_{12}\equiv ( \OO_{C_1\cup C_2}(a_1, a_2), \OO_{C_2}(a_2-2)).$

In Subcase 4.1, arguing as Subcase 3.1, we have
\begin{align*}
\RR
\equiv {}&(\OO_{C_2}(a_2), \OO_{C_1\cup C_2}(a_1, a_2), \OO_{C_3}(a_3), \OO_{C_3}(b_3), \OO_{C_3}(c_3))\\
\equiv {}&(\OO_{C_2}(a_2), \OO_{C_1\cup C_2\cup C_3 }(a_1, a_2+1, a_3),  \OO_{C_3}(b_3), \OO_{C_3}(c_3))\\
\equiv {}&(\OO_{C_1\cup C_2\cup C_3 }(a_1, a_2+1, a_3), \OO_{C_2}(a_2),  \OO_{C_3}(b_3), \OO_{C_3}(c_3))\\
\equiv {}&(\OO_{C_1\cup C_2\cup C_3 }(a_1, a_2+1, a_3),\OO_{C_2\cup C_3}(a_2+1, b_3), \OO_{C_3}(c_3)).
\end{align*}
This gives (4-1) by changing $a_2$ appropriately.

In Subcase 4.2, arguing as Subcase 3.2, we have perfect factorizations
\begin{align*}
\RR
\equiv {}&(\OO_{C_1\cup C_2}(a_1, a_2), \OO_{C_2}(a_2-1), \OO_{C_3}(a_3), \OO_{C_3}(b_3), \OO_{C_3}(c_3))\\
\equiv {}&(\OO_{C_1\cup C_2}(a_1, a_2), \OO_{C_2\cup C_3}(a_2, a_3),  \OO_{C_3}(b_3), \OO_{C_3}(c_3))\\
\equiv {}&(\OO_{C_2\cup C_3}(a_2, a_3), \OO_{C_1\cup C_2}(a_1, a_2),  \OO_{C_3}(b_3), \OO_{C_3}(c_3))\\
\equiv {}&(\OO_{C_2\cup C_3}(a_2, a_3),\OO_{C_1\cup C_2\cup C_3}(a_1, a_2+1, b_3),   \OO_{C_3}(c_3)).
\end{align*}
This gives (4-2).

In Subcase 4.3, arguing as Subcase 3.3, we have perfect factorizations
 \begin{align*}
\RR
\equiv {}&(\OO_{C_1\cup C_2}(a_1, a_2), \OO_{C_2}(a_2-2), \OO_{C_3}(a_3), \OO_{C_3}(b_3), \OO_{C_3}(c_3))\\
\equiv {}&(\OO_{C_1\cup C_2}(a_1, a_2), \OO_{C_2\cup C_3}(a_2-1, a_3),  \OO_{C_3}(b_3), \OO_{C_3}(c_3))\\
\equiv {}&(\OO_{C_1\cup 2C_2\cup C_3}(a_1+1, a_2-1, a_3),  \OO_{C_3}(b_3), \OO_{C_3}(c_3)).
\end{align*}

If $a_3>b_3+1$, then $a_3=b_3+2$ in this case since the extension of $\OO_{C_3}(b_3)$ by $\OO_{C_3}(a_3)$ is rigid (see Case 2 of proof of Proposition \ref{prop 12}). Arguing as Subcase 3.3, we have perfect factorizations
  \begin{align*}
\RR\equiv {}&( \OO_{C_1\cup C_2\cup C_3}(a_1, a_2+1, a_3-1), \OO_{C_2}(a_2-1), \OO_{C_3}(b_3), \OO_{C_3}(c_3))\\
\equiv {}&( \OO_{C_1\cup C_2\cup C_3}(a_1, a_2+1, b_3+1), \OO_{C_2\cup C_3}(a_2, b_3), \OO_{C_3}(c_3)).
\end{align*}
This gives (4-3) by changing $a_2$ appropriately.

If $a_3=b_3+1$, then as Subcase 3.3, we have a perfect factorization
  \begin{align*}
\RR\equiv {}&( \OO_{C_1\cup 2C_2\cup C_3}(a_1+1, a_2-1, b_3+1)\oplus \OO_{C_3}(b_3), \OO_{C_3}(c_3))
\end{align*}
since 
$$
h^1(  \OO_{C_3}(b_3),\OO_{C_1\cup 2C_2\cup C_3}(a_1+1, a_2-1, b_3+1))=0.
$$
This gives (4-4) by changing $a_1, a_2$ appropriately.
 \end{proof}

We get the following corollary directly.
\begin{cor}\label{cor 123}
Let $C_1\cup C_2\cup C_3$ be a chain of three $(-2)$-curves and $\RR$ a torsion rigid sheaf with $c_1(\RR)=C_1+2C_2+3C_3$. Then one of the following holds
\begin{enumerate}
\item  $\RR$ has a perfect factorization $(\GG, \LL)$
where $\LL$ is a line bundle supported on the chain $C_i\cup \cdots \cup C_3$ for some $1\leq i \leq 3$, or
\item
$\RR\cong \OO_{C_1\cup 2C_2 \cup C_3}(a_1, a_2, a_3)\oplus \OO_{C_3}(a_3-1)^{\oplus 2}$ for some integers  $a_1, a_2, a_3$.
\end{enumerate}
\end{cor}

\begin{cor}\label{cor 12321}
Let $C_1\cup \cdots \cup C_5$ be a chain of five $(-2)$-curves and $\RR$ a torsion rigid sheaf with $c_1(\RR)=C_1+2C_2+3C_3+2C_4+C_5$. Then $\RR$ has a perfect factorization $(\GG, \LL)$
where $\LL$ is a line bundle supported on either the chain $C_i\cup \cdots \cup C_3$ for some $1\leq i \leq 5$, or  the chain $C_1\cup \cdots \cup C_5$.
\end{cor}
\begin{proof}
Taking the restriction to $C_1\cup C_2\cup C_3$ and  $C_3\cup C_4\cup C_5$,  we have exact sequences
\begin{align*}
0\to \RR_{45}\to \RR \to \RR_{123}\to 0,\\
0\to \RR_{12}\to \RR \to \RR_{345}\to 0.
\end{align*}
Note that $c_1(\RR_{123})=C_1+2C_2+3C_3$ and $c_1(\RR_{345})=3C_3+2C_4+C_5$.
If one of  $\RR_{123}$ and  $\RR_{345}$ satisfies Corollary \ref{cor 123}(1), then we can get the desired perfect factorization.

Suppose that both $\RR_{123}$ and  $\RR_{345}$ satisfy Corollary \ref{cor 123}(2), note that their restriction on $C_3$ are the same, for simplicity and without loss of generality, we may write
\begin{align*}
\RR_{123}{}&\cong \OO_{C_1\cup 2C_2 \cup C_3}\oplus \OO_{C_3}(-1)^{\oplus 2},\\
\RR_{345}{}&\cong \OO_{C_3\cup 2C_4 \cup C_5}\oplus \OO_{C_3}(-1)^{\oplus 2}.
\end{align*}
In this case we have an exact sequence 
$$
0\to \OO_{C_1\cup 2C_2}(0, -1) \oplus\OO_{2C_4\cup C_5}(-1, 0)\to \RR\to\OO_{C_3}\oplus \OO_{C_3}(-1)^{\oplus 2}\to 0.
$$
This gives perfect factorizations 
\begin{align*}
\RR\equiv{}& (\OO_{C_1\cup 2C_2}(0, -1) \oplus\OO_{2C_4\cup C_5}(-1, 0), \OO_{C_3}, \OO_{C_3}(-1)^{\oplus 2})\\
\equiv{}& (\OO_{C_2} \oplus\OO_{C_4}, \OO_{C_1\cup C_2}(0, -1) \oplus\OO_{C_4\cup C_5}(-1, 0), \OO_{C_3}, \OO_{C_3}(-1)^{\oplus 2})\\
\equiv{}& (\OO_{C_2} \oplus\OO_{C_4}, \OO_{C_1\cup C_2}(0, -1), \OO_{C_4\cup C_5}(-1, 0), \OO_{C_3}, \OO_{C_3}(-1)^{\oplus 2})\\
\equiv{}& (\OO_{C_2} \oplus\OO_{C_4}, \OO_{C_1\cup C_2}(0, -1), \OO_{C_3\cup C_4\cup C_5}, \OO_{C_3}(-1)^{\oplus 2})\\
\equiv{}& (\OO_{C_2} \oplus\OO_{C_4}, \OO_{C_1\cup C_2\cup C_3\cup C_4\cup C_5}, \OO_{C_3}(-1)^{\oplus 2}).
\end{align*}
Here we apply Lemma \ref{lem GSH} in the last two steps. Note that 
$$
h^1( \OO_{C_3}(-1), \OO_{C_1\cup C_2\cup C_3\cup C_4\cup C_5})=0
$$
by computing Hom's and $\chi$.
Hence we exchange the last two factors and get a perfect factorization
$$
\RR\equiv (\OO_{C_2} \oplus\OO_{C_4}, \OO_{C_3}(-1)^{\oplus 2},  \OO_{C_1\cup C_2\cup C_3\cup C_4\cup C_5}),
$$
and the proof is completed.
\end{proof}

\begin{cor}\label{cor 123321}
Let $C_1\cup \cdots \cup C_6$ be a chain of six $(-2)$-curves and $\RR$ a torsion rigid sheaf with $c_1(\RR)=C_1+2C_2+3C_3+3C_4+2C_5+C_6$. Then $\RR$ has a perfect factorization $(\GG, \LL)$
where $\LL$ is a line bundle supported on one of  the following chains:
\begin{enumerate} 
\item $C_i\cup \cdots \cup C_3$ for some $1\leq i \leq 3$;
\item  $C_2\cup \cdots \cup C_j$ for some $4\leq j \leq 6$;
\item $C_3\cup C_4$.
\end{enumerate}
\end{cor}
\begin{proof}
Taking the restriction to $C_1\cup C_2\cup C_3$, we have an exact sequence
\begin{align*}
0\to \RR_{456}\to \RR \to \RR_{123}\to 0.
\end{align*}
Note that $c_1(\RR_{123})=C_1+2C_2+3C_3$.
If $\RR_{123}$ satisfies Corollary \ref{cor 123}(1), then we  get the first case.

Now suppose that $\RR_{123}$ satisfies Corollary \ref{cor 123}(2). For simplicity and without loss of generality, we may assume that
$$
\RR_{123}\cong \OO_{C_1\cup 2C_2 \cup C_3}\oplus \OO_{C_3}(-1)^{\oplus 2},
$$
and we have a perfect factorization
$$
\RR\equiv (\RR_{456}, \OO_{C_1\cup C_2}(-1, 1),\OO_{C_2\cup C_3}, \OO_{C_3}(-1)^{\oplus 2}).
$$

On the other hand, $c_1(\RR_{456})=3C_4+2C_5+C_6$.

Suppose that $\RR_{456}$ has a perfect factorization $(\GG', \LL')$
where $\LL'$ is a line bundle supported on the chain $C_4\cup \cdots \cup C_j$ for some $4\leq j \leq 6$. For simplicity and without loss of generality, we may assume that $\LL'\cong \OO_{C_4\cup \cdots \cup C_j}$. Hence  we have  perfect factorizations
\begin{align*}
\RR{}& \equiv (\GG', \OO_{C_4\cup \cdots \cup C_j}, \OO_{C_1\cup C_2}(-1, 1),\OO_{C_2\cup C_3}, \OO_{C_3}(-1)^{\oplus 2})\\
{}& \equiv (\GG',  \OO_{C_1\cup C_2}(-1, 1),\OO_{C_4\cup \cdots \cup C_j},\OO_{C_2\cup C_3}, \OO_{C_3}(-1)^{\oplus 2})\\
{}& \equiv (\GG',  \OO_{C_1\cup C_2}(-1, 1),\OO_{C_2\cup C_3\cup C_4\cup \cdots \cup C_j}(0,0,1,\ldots), \OO_{C_3}(-1)^{\oplus 2})\\
{}& \equiv (\GG',  \OO_{C_1\cup C_2}(-1, 1),\OO_{C_3}(-1)^{\oplus 2}, \OO_{C_2\cup C_3\cup C_4\cup \cdots \cup C_j}(0,0,1,\ldots)).
\end{align*}
We apply Lemma \ref{lem GSH} in the second step, and the last step is because
$$
h^1(\OO_{C_3}(-1), \OO_{C_2\cup C_3\cup C_4\cup \cdots \cup C_j}(0,0,1,\ldots))=0
$$
by computing Hom's and $\chi$. This gives the second case of this corollary.

Now suppose that $\RR_{456}$ satisfies Corollary \ref{cor 123}.  For simplicity and without loss of generality, we may assume that
$$
\RR_{456}\cong \OO_{C_4\cup 2C_5 \cup C_6}\oplus \OO_{C_4}(-1)^{\oplus 2}.
$$
Then we have  perfect factorizations
\begin{align*}
\RR{}& \equiv (\OO_{C_4\cup 2C_5 \cup C_6}, \OO_{C_4}(-1)^{\oplus 2}, \OO_{C_1\cup C_2}(-1, 1),\OO_{C_2\cup C_3}, \OO_{C_3}(-1)^{\oplus 2})\\
{}& \equiv (\OO_{C_4\cup 2C_5 \cup C_6}, \OO_{C_1\cup C_2}(-1, 1),\OO_{C_4}(-1)^{\oplus 2}, \OO_{C_2\cup C_3}, \OO_{C_3}(-1)^{\oplus 2})\\
{}& \equiv (\OO_{C_4\cup 2C_5 \cup C_6}, \OO_{C_1\cup C_2}(-1, 1),  \OO_{C_2\cup C_3\cup{C_4}},\OO_{C_4}(-1), \OO_{C_3}(-1)^{\oplus 2})\\
{}& \equiv (\OO_{C_4\cup 2C_5 \cup C_6}, \OO_{C_1\cup C_2}(-1, 1),  \OO_{C_2\cup C_3\cup{C_4}},\OO_{C_3}(-1), \OO_{C_3\cup C_4}(-1, 0)).
\end{align*}
Here we apply Lemma \ref{lem GSH} in the last two steps. This gives the third case of this corollary.
\end{proof}

\section{Classification of torsion exceptional sheaves}
\noindent In this section, we prove Theorems \ref{main1} and \ref{main2}.

\begin{lem}\label{lemma main2}
Let $\EE$ be a torsion exceptional sheaf on a smooth projective surface $X$ satisfying conditions in Theorem \ref{main2}. Assume that there exists at least one $(-2)$-curve in  $\supp(\EE)$. Then there exists a chain of $(-2)$-curves $Z$ in $\supp(\EE)$, and a line bundle $\LL$ on $Z$, such that $c_1(\EE)\cdot c_1(\LL)=-1$ and there is an exact sequence
$$
0\to \EE'\to \EE\to \LL\to 0
$$
with $h^0(\EE', \LL)=0$.
\end{lem}

\begin{proof}
By assumption, 
we may write
$$\supp(\EE)=D\cup\bigcup_{j=1}^m\bigcup_{i=1}^{n_j}C^j_i, $$
where $C^j_1\cup \cdots \cup C^j_{n_j}$ is a chain of $(-2)$-curves for each $j$ and they are disjoint from each other. Since $\EE$ is exceptional, $\supp(\EE)$ is connected. Hence we assume that $D$ intersects with the chain $C^j_1\cup \cdots \cup C^j_{n_j}$ on the curve $C^j_{k_j}$ at one point for each $j$.
We may write
$$c_1(\EE)=D+\sum_{j=1}^m\sum_{i=1}^{n_j}r^j_i C^j_i, $$
in the sense of Remark \ref{remark restriction} since the first Chern class of the restriction to every chain of $(-2)$-curves is uniquely determined.
Since $\EE$ is exceptional, by Riemann--Roch formula, 
$c_1(\EE)^2=-\chi(\EE, \EE)=-1.$
On the other hand,
\begin{align*}
c_1(\EE)^2{}&=\Big(D+\sum_{j=1}^m\sum_{i=1}^{n_j}r^j_i C^j_i\Big)^2\\
{}&=D^2+2D\cdot \sum_{j=1}^m\sum_{i=1}^{n_j}r^j_i C^j_i+\Big(\sum_{j=1}^m\sum_{i=1}^{n_j}r^j_i C^j_i\Big)^2\\
{}&=D^2+2D\cdot \sum_{j=1}^m\sum_{i=1}^{n_j}r^j_i C^j_i+\sum_{j=1}^m\Big(\sum_{i=1}^{n_j}r^j_i C^j_i\Big)^2\\
{}&=-1+2\sum_{j=1}^mr^j_{k_j}+\sum_{j=1}^m\Big(-2\sum_{i=1}^{n_j}(r^j_i)^2+2\sum_{i=1}^{n_j-1}r^j_ir^j_{i+1}\Big).
\end{align*}
This implies that
$
\sum_{j=1}^m f(r^j_1, \ldots, r^j_{n_j}; k_j)=0,
$
where $f$ is the polynomial defined in Subsection \ref{polynomial}. By Proposition \ref{prop polynomial},  $f(r^j_1, \ldots, r^j_{n_j}; k_j)=0$ for each $j$ and $\{r^j_1, \ldots, r^j_{n_j}, k_j\}$ satisfies the conditions in Proposition \ref{prop polynomial}.

For convenience, we write $n_1=n$, $r_i^1=r_i$, $C^1_i=C_i$, $k_1=k$.
Note that $n\leq 6$ by assumption,  and hence $r_k\leq 3$.

Reversing the order of $\{C_i\}$ if necessary, by the conditions in Proposition \ref{prop polynomial}, we only have the following 6 cases:\begin{enumerate}
\item $k=n=1$, $r_1=1$;
\item $k\geq 2$ and $r_1=r_2=1$;
\item $k=2$ and $r_1=1, r_2=2, r_3=1$;
\item $k\geq 3$ and $r_1=1, r_2=r_3=2$;
\item $k=3$, $n=5$ and $(r_1, \ldots, r_5)=(1,2,3,2,1)$;
\item $k=4$, $n=6$ and $(r_1, \ldots, r_6)=(1,2,3,3,2,1)$.
\end{enumerate}

In Case (1) and (2), taking the restriction to $C_1$,
we have an exact sequence
$$
0\to \EE'\to \EE \to \RR_1\to 0.
$$
Then $c_1(\RR_1)=C_1$ and hence $\RR_1$ is a line bundle on $C_1$. Moreover, $h^0(\EE', \RR_1)=0$ by construction, and 
$$
c_1(\EE)\cdot c_1(\RR_1)=
\begin{cases}(C_1+D)\cdot C_1=-1 & \text{Case (1)};\\
(C_1+C_2)\cdot C_1=-1  & \text{Case (2)}.
\end{cases}
$$
Hence we may take $\LL=\RR_1$.

In Case (3) and (4), taking the restriction to $C_1\cup C_2$,
we have an exact sequence
$$
0\to \EE_{12}\to \EE \to \RR_{12}\to 0.
$$
Then $c_1(\RR_{12})=C_1+2C_2$. By Proposition \ref{prop 12}, 
there exists a line bundle $\LL$ supported on $C_2$ or the chain $C_1 \cup C_2$ with an exact sequence 
$$
0\to \GG\to \RR_{12}\to \LL\to 0
$$
such that $h^0(\GG, \LL)=0$. Consider the exact sequence
$$
0\to \EE'\to \EE\to \LL\to 0
$$ given by the surjection $\EE\to  \RR_{12}\to\LL$. Then $h^0(\EE', \LL)=0$ since $\EE'$ is an extension of $\GG$ by $\EE_{12}$ and $\supp(\EE_{12})$ does not contain $C_1$ or $C_2$. Note that $c_1(\LL)=C_2$ or $C_1+C_2$, we have 
$$
c_1(\EE)\cdot c_1(\LL)=
\begin{cases}(C_1+2C_2+C_3+D)\cdot c_1(\LL)=-1 & \text{Case (3)};\\
(C_1+2C_2+2C_3)\cdot c_1(\LL)=-1  & \text{Case (4)}.
\end{cases}
$$
This $\LL$ satisfies all conditions we require.

In Case (5), taking the restriction to $C_1\cup \cdots \cup C_5$,
we have an exact sequence
$$
0\to \EE_{1}\to \EE \to \RR\to 0.
$$
Then $c_1(\RR)=C_1+2C_2+3C_3+2C_4+C_5$. By Corollary \ref{cor 12321}, 
$\RR$ has a perfect factorization $(\GG, \LL)$
where $\LL$ is a line bundle supported on either the chain $C_i\cup \cdots \cup C_3$ for some $1\leq i \leq 5$, or  the chain $C_1\cup \cdots \cup C_5$.
This induces an exact sequence
$$
0\to \EE'\to \EE\to \LL\to 0
$$
where $\EE'$ is an extension of $\GG$ by $\EE_{1}$. In particular, we have $h^0(\EE', \LL)=0$. By construction, $c_1(\LL)=\sum_{j=i}^3C_j$ for some $1\leq i \leq 5$ or $\sum_{j=1}^5C_j$. Note that $D$ only intersects with $C_3$, it is easy to compute that 
$$
c_1(\EE)\cdot c_1(\LL)=(C_1+2C_2+3C_3+2C_4+C_5+D)\cdot c_1(\LL)=-1.
$$
This $\LL$ satisfies all conditions we require.

In Case (6), taking the restriction to $C_1\cup \cdots \cup C_6$,
we have an exact sequence
$$
0\to \EE_{1}\to \EE \to \RR\to 0.
$$
Then $c_1(\RR)=C_1+2C_2+3C_3+3C_4+2C_5+C_6$. By Corollary \ref{cor 123321}, 
$\RR$ has a perfect factorization $(\GG, \LL)$
where $\LL$ is a line bundle supported on the chain \item $C_i\cup \cdots \cup C_3$ for some $1\leq i \leq 3$, or the chain $C_2\cup \cdots \cup C_j$ for some $4\leq j \leq 6$, or the chain $C_3\cup C_4$.
This induces an exact sequence
$$
0\to \EE'\to \EE\to \LL\to 0
$$
where $\EE'$ is an extension of $\GG$ by $\EE_{1}$. In particular, we have $h^0(\EE', \LL)=0$. By construction, $c_1(\LL)=\sum_{l=i}^3C_l$ for some $1\leq i \leq 3$, or $\sum_{l=2}^jC_l$ for some $4\leq j \leq 6$, or $C_3+C_4$. Note that $D$ only intersects with $C_4$, it is easy to compute that 
$$
c_1(\EE)\cdot c_1(\LL)=(C_1+2C_2+3C_3+3C_4+2C_5+C_6+D)\cdot c_1(\LL)=-1.
$$
This $\LL$ satisfies all conditions we require.
\end{proof}

\begin{proof}[Proof of Theorem \ref{main2}]
As in the proof of Lemma \ref{lemma main2}, we may write
$$c_1(\EE)=D+\sum_{j=1}^m\sum_{i=1}^{n_j}r^j_i C^j_i, $$
where $C^j_1\cup \cdots \cup C^j_{n_j}$ is a chain of $(-2)$-curves for each $j$ and they are disjoint from each other.  

Assume that there exists at least one $(-2)$-curve in  $\supp(\EE)$, then by Lemma \ref{lemma main2} there exists a chain of $(-2)$-curves $Z$ in $\supp(\EE)$, and a line bundle $\LL$ on $Z$, such that $c_1(\EE)\cdot c_1(\LL)=-1$ and there is an exact sequence
$$
0\to \EE'\to \EE\to \LL\to 0
$$
with $h^0(\EE', \LL)=0$.
Note that $\LL$ is a spherical object, and 
$$
\chi(\LL, \EE')=\chi(\LL, \EE)-\chi(\LL, \LL)=-c_1(\EE)\cdot c_1(\LL)-2=-1.
$$
By Lemma \ref{lem EES}, $\EE'$ is exceptional and $\EE\cong T_{\LL}\EE'$. Moreover, by the proof of Lemma \ref{lemma main2},
$$c_1(\EE')=c_1(\EE)-c_1(\LL)=D+\sum_{j=1}^m\sum_{i=1}^{n_j}(r^j_i)' C^j_i,, $$
where $(r^j_i)'=\begin{cases}
r^j_i-1 & \text{ if }C^j_i\subset \supp(\LL);\\
   r^j_i    & \text{ otherwise}.
\end{cases}$

By induction on the number $\sum_{j=1}^m\sum_{i=1}^{n_j}r^j_i$, after finitely many steps, we may assume that $c_1(\EE)=D$. This implies that $\EE$ is a line bundle on $D$ and the proof is completed.
\end{proof}

\begin{lem}\label{lem D in E}
Let $\EE$ be a torsion exceptional sheaf on a weak del Pezzo surface $X$, then there exists exactly one $(-1)$-curve $D$ in $\supp(\EE)$, and the restriction of $\EE$ in $D$ is a line bundle.
\end{lem}

\begin{proof}
Since $|-K_X|$ has no base component by Lemma \ref{lem no base}, choose a general element in $E\in |-K_X|$ which is not contained in $\supp(\EE)$. There is a short exact sequence
$$0\rightarrow \omega _X \rightarrow \OO_X\rightarrow \OO_E \rightarrow 0.$$
Tensoring with $\EE$, since $\EE$ is pure one-dimensional and $E\not \subset\supp(\EE)$, we get an exact sequence
\begin{align*} 
0\rightarrow \EE\otimes \omega_X \rightarrow \EE \rightarrow \EE|_E\rightarrow 0.
\end{align*}
Applying $\Hom(\EE,-)$ to this sequence, we get an exact sequence 
$$
\Hom(\EE,\EE\otimes \omega_X)\rightarrow \Hom(\EE,\EE)\rightarrow \Hom(\EE,\EE|_E)
\rightarrow \Ext^1(\EE,\EE\otimes \omega_X).
$$
Since $\EE$ is exceptional,  $h^0(\EE,\EE)=1$ and $h^0(\EE,\EE\otimes \omega_X)=h^1(\EE,\EE\otimes \omega_X)=0$ by Serre duality. Hence
$
\Hom(\EE,\EE|_E)\cong\mathbb{C}.
$
By Lemma \ref{lem supp}, $\supp(\EE)$ only contains $(-1)$-curves and $(-2)$-curves.
Note that each $(-1)$-curve intersects with $E$ at one point and each $(-2)$-curve does not intersect with $E$,
we conclude that there is only one $(-1)$-curve $D$ in $\supp(\EE)$. 
Moreover, taking restriction to $D$, we get an exact sequence
$$
0\to \EE'\to \EE\to\EE_D\to 0,
$$
where the support of $\EE'$ only contains $(-2)$-curves.
Combining with the fact that $
\Hom(\EE,\EE|_E)\cong\mathbb{C}
$, we have  
$\Hom(\EE_D,\EE_D|_E)\cong\mathbb{C}.$
Since $\EE_D$ is pure one-dimensional, $\EE_D$ is a line bundle on $D$.
\end{proof}

\begin{proof}[Proof of Theorem \ref{main1}]
It suffices to check that any torsion exceptional sheaf $\EE$ on a weak del Pezzo surface $X$ of $d>2$ of Type A satisfies conditions (1)-(4) in Theorem \ref{main2}.

By  Lemma \ref{lem supp}, any irreducible component of $\supp(\EE)$ is a curve with negative self-intersection, hence is a $(-1)$-curve or $(-2)$-curve. By Lemma \ref{lem D in E}, conditions (1)--(2) are satisfied.
By the assumption $d>2$ and Lemma \ref{lem no of -2}, there are at most $6$ $(-2)$-curves on $X$, hence condition (3) is satisfied since $X$ is of Type A. Again by the assumption $d>2$ and Lemma \ref{lem d>1}, condition (4) is satisfied.
\end{proof}

\section*{Acknowledgment}
\noindent The first author would like to express his deep gratitude to his supervisor Professor Yujiro Kawamata for discussion and warm encouragement. The authors are grateful to the anonymous refree(s) for valuable comments and suggestions.

\end{document}